\documentclass[reqno]{amsart}

\usepackage{amssymb}
\usepackage{amsthm}
\usepackage{amsmath}
\usepackage{float}
\usepackage{fancyhdr}
\usepackage{longtable}

\usepackage{tikz-cd}

\usepackage[matrix, arrow]{xy}
\xyoption{arrow}
\xyoption{curve}
\theoremstyle{plain}\newtheorem{Theorem}{Theorem}[section]
\theoremstyle{plain}
\theoremstyle{plain}\newtheorem{Corollary}[Theorem]{Corollary}
\theoremstyle{plain}\newtheorem{Lemma}[Theorem]{Lemma}
\theoremstyle{plain}\newtheorem{Proposition}[Theorem]{Proposition}
\theoremstyle{plain}
\theoremstyle{definition}\newtheorem{Definition}[Theorem]{Definition}
\theoremstyle{definition}
\theoremstyle{definition}
\theoremstyle{definition}
\theoremstyle{definition}
\theoremstyle{definition}\newtheorem{Remark}[Theorem]{Remark}
\theoremstyle{definition}
\theoremstyle{definition}

                 \def\ten{\otimes}

\def\dim{\mathrm{dim}}

\def\Ext{\mathrm{Ext}}

\def\Hom{\mathrm{Hom}}

\def\ker{\mathrm{ker}}

\def\op{\mathrm{op}}

\def\lmd{\lambda}

\usepackage{tikz-cd}

\def\hx{\hat{x}}
\def\hy{\hat{y}}

\def\hg{\hat{g}}
\def\hh{\hat{h}}
\def\ha{\hat{a}}

\def\hG{\hat{G}}

\def\lmd{\lambda}

\usepackage[hidelinks,bookmarks=false]{hyperref}
\usepackage{tikz-cd}

\def\hx{\hat{x}}
\def\hy{\hat{y}}

\def\hg{\hat{g}}
\def\hh{\hat{h}}
\def\ha{\hat{a}}

\def\hG{\hat{G}}

\def\kax{k_\alpha\hx}

\newtheorem*{theorem*}{Theorem}

\usepackage{caption}
\makeindex
\author{William Murphy} 
\date{\today}

\begin{document}

\title[The Nonvanishing first Hochschild cohomology of twisted simple group algebras]{The Nonvanishing first Hochschild cohomology of twisted finite simple group algebras}

\begin{abstract}
     Let $G$ be a finite simple group and $k$ be an algebraically closed field of prime characteristic dividing the order of $G$. We show that for all $2$-cocycles $\alpha\in Z^2(G;k^\times)$, the first Hochschild cohomology group of the twisted group algebra $HH^1(k_\alpha G)$ is nonzero.
\end{abstract}

\maketitle
\vspace{-7mm}
\section{Introduction}\label{sec1}

Throughout, let $k$ be an algebraically closed field of prime characteristic $p$. For a finite group $G$ of order divisible by $p$, the nonvanishing of the first Hochschild cohomology of the group algebra $kG$ is a consequence of a result of Fleischmann, Janiszczak and Lempken \cite{FJL}. There, it is shown that one can always find an element $x\in G$ of order divisible by $p$, whose $p$-part is not contained in $C_G(x)'$, the commutator subgroup of the centraliser of $x$ in $G$; developing the terminology used in \cite{FJL} we will call such an element $x$ a \textit{weak Non-Schur element}. Consequently, one obtains the nontriviality of $\Hom(C_G(x),k)$ for $x$ a weak Non-Schur element, and combining this with the centraliser decomposition for the Hochschild cohomology of group algebras (\cite[Theorem 2.11.2]{BensonII}), one sees that $HH^1(kG)$ is always nonzero. 

On the other hand, this is still an open problem for the family of \textit{twisted group algebras}, that is, the algebras obtained by ``twisting" the group algebra multiplication by some  $2$-cocycle. Formally, let $\alpha\in Z^2(G;k^\times)$ and denote by $k_\alpha G$ the $k$-algebra with a $k$-basis given by $X:=\{\hx \ | \ x\in G\}$ and multiplication given by $\hx\hy=\alpha(x,y)\widehat{xy}$ for all $\hx,\hy\in X$. Here, the product $\hx\hy$ is in $k_\alpha G$ whilst $\widehat{xy}$ is the image in $X$ of the product $xy$ in $G$.

In this note we determine the following result, similar to that of \cite{FJL}, which may be of independent interest.  

\begin{Theorem}\label{main}
Let $G$ be a finite simple group of order divisible by $p$. Then $G$ contains a weak Non-Schur element $x$ which satisfies $\alpha(g,x)=\alpha(x,g)$  for all $\alpha\in Z^2(G;k^\times)$ and all $g\in C_G(x)$.
\end{Theorem}

\noindent Our motivation for this result lies in one immediate consequence of this theorem, as we obtain the nonvanishing first Hochschild cohomology of twisted finite simple group algebras.

\begin{Corollary}\label{mainnonvanishing}
Let $G$ be a finite simple group of order divisible by $p$. Then for all $2$-cocycles $\alpha\in Z^2(G;k^\times)$, the first Hochschild cohomology group of the twisted group algebra $HH^1(k_\alpha G)\neq \{0\}$.
\end{Corollary}

Just as with the result of Fleischmann, Janiszczak and Lempken, the proof of Theorem \ref{main} uses the classification of finite simple groups. What is more it is interesting to note that no proof of the nonvanishing first Hochschild cohomology of group algebras is known at present, that does not use the classification. We are also unaware of a reduction to finite simple groups for the nonvanishing Hochschild cohomology of twisted group algebras.

There has been some recent independent interest in the first Hochschild cohomology of twisted group algebras from Todea \cite{Tod}. As with our work, Todea takes \cite{FJL} as a starting point and some similar methods are used, though the main results differ: there, it is shown that for $p$-solvable groups, their twisted group algebras have nonzero first Hochschild cohomology. Consequently, we are able to give a partial answer to a question of Todea (\cite[Question 1.3]{Tod}) in the affirmative, with a reduction to finite simple groups remaining the only barrier to a full affirmative answer.

In addition to generalising the case where $\alpha$ represents the trivial class in $H^2(G;k^\times)$, the nonvanishing of the first Hochschild cohomology of twisted group algebras is a problem of interest in its own right as may be seen when one observes that $k_\alpha G$ always arises as the direct sum of blocks of some central group extension of $G$ (see \cite[\S1]{Humph}). Thus showing that $HH^1(k_\alpha G)\neq\{0\}$ in general can also be seen as a stepping stone towards proving the nonvanishing of $HH^1(B)$ for a block $B$ of a finite group algebra with a nontrivial defect group.

The first Hochschild cohomology of a $k$-algebra $A$ admits a Lie algebra structure, which extends to a graded Lie algebra structure on $HH^*(A)$, though in this note we are only interested in degree one. Our motivation to study the more general problem of the nonvanishing of $HH^1(B)$ comes in part from a wider investigation into the links between the Lie algebra structure of $HH^1(B)$ and the $k$-algebra structure of $B$, examining how each influences the other. In particular, the Lie algebra structure is also expected to provide information useful to the Auslander-Reiten conjecture (see \cite{BenKesLinBV,BriRub,ChapSchSol,EisRad,LinRubI,LinRubII,MurMat,RubSchroSol} for more examples of this).

The proof of Theorem \ref{main} amounts to an examination of the weak Non-Schur elements given in the Fleischmann, Janiszczak and Lempken result \cite{FJL}, and determining that in each case the desired property holds. In Section \ref{sec2} we give auxiliary results and details of twisted group algebras and Hochschild cohomology in the first degree, and in Section \ref{sec3} we separate the main result into propositions, showcasing different methods to prove that it holds for specific families of finite simple groups. Finally, in Section \ref{sec4} we use \texttt{GAP} to calculate some dimensions of Hochschild cohomology groups of twisted finite simple group algebras.

\section{Auxiliary results}\label{sec2}

Throughout this section we fix $G$ to be an arbitrary finite group unless otherwise stated. For a $2$-cocycle $\alpha \in Z^2(G;k^\times)$ the associativity of multiplication in $k_\alpha G$ is a consequence of the $2$-cocycle identity. Recall that we have an isomorphism of abelian groups between $H^2(G;k^\times)$ and $M(G)_{p'}$, the component of the Schur multiplier of $G$ of order coprime to $p$. It is well known that there is a bijective correspondence between the cohomology classes of $2$-cocycles $[\alpha]\in H^2(G;k^\times)$ and central extensions $1\to k^\times\to H\to G\to 1$. Moreover, for such an $\alpha$ one can always reduce to a central extension of the form $1\to Z\to \hG\to G \to 1,$ for some finite subgroup $\hG$ of  $k^\times$ and cyclic $p'$-subgroup $Z$ of $k^\times$ (see for example \cite[Proposition 1.2.18]{LinBlockI}). With this notation, multiplication in $\hG$ is equivalent to multiplication between elements of the basis of $k_\alpha G$ indexed by $G$, and so we make no distinction between the two. In particular, for each $x\in G$, we write $\hx$ to mean both the image of $x$ in $k_\alpha G$ or a preimage of $x$ in $\hG$. Note that we may now restrict to the case where $\alpha$ takes values in $Z$, and these details are made explicit by the following fact (\cite[Theorem 1.1]{Humph}): with $|Z|=m$ we have an isomorphism of $k$-algebras \begin{equation}
\label{dirprod}
k\hG\cong \prod_{i=0}^{m-1} k_{\alpha^i}G.
\end{equation} 

Let $A$ be an associative $k$-algebra and let $M$ be an $(A,A)$-bimodule. Let $A^e$ denote the tensor product of $k$-algebras $A\ten_k A^{\op}$ and view $M$ as an $A^e$-module. Then for all $n\geq 0$ the $n$'th Hochschild cohomology group of $A$ with coefficients in $M$ is the $k$-module  $$HH^n(A;M)=\Ext^n_{A^e}(A;M),$$ and the $n$'th Hochschild cohomology group of $A$ is $HH^n(A;A)=HH^n(A)$.

It is straightforward to show that taking Hochschild cohomology of (twisted) group algebras distributes nicely with respect to $(\ref{dirprod})$, and we obtain an isomorphism of $k$-modules, \begin{equation}
\label{dirsum}
HH^1(k\hG)\cong \bigoplus_{i=0}^{m-1} HH^1(k_{\alpha^i}G).
\end{equation}\label{(2)} This will see regular applications in the proof of Proposition \ref{dimsprop}: in particular it will be used to calculate the dimensions of the degree one Hochschild cohomology of twisted simple group algebras for small $Z$, and one notes that we have \begin{equation}
\label{dimsum}
\dim_k(HH^1(k\hG))-\dim_k(HH^1(kG))=\sum_{i=1}^{m-1}\dim_k(HH^1(k_{\alpha^i}G)).\end{equation} What is more, it is easily seen that $k_{\alpha^{-1}}G\cong (k_\alpha G)^{\op}$ as $k$-algebras and $HH^n(A)\cong HH^n(A^\op)$ as Lie algebras for an arbitrary $k$-algebra $A$ and for all $n$, which can be used to simplify (\ref{dirsum}) and (\ref{dimsum}) in explicit cases, as this results in an equality of dimensions $\dim_k(HH^1(k_\alpha G))=\dim_k(HH^1(k_{\alpha^{-1}}G))$. Moreover, the \textit{centraliser decomposition for the Hochschild cohomology of group algebras} (\cite[Theorem 2.11.2]{BensonII}) allows for the easy computation of the dimensions on the left hand side of (\ref{dimsum}).

The centraliser decomposition has a generalisation to twisted group algebras, which we will refer to throughout.

\begin{Proposition}\cite[Lemma 3.5]{WithSpoon}\label{CD}
Let $G$ be a finite group and $[\alpha]\in H^2(G;k^\times)$. Then there is a canonical isomorphism of graded $k$-modules $$HH^*(k_\alpha G)\cong \bigoplus_x H^*(C_G(x);\kax),$$ where $x$ runs over a complete set of conjugacy class representatives of $G$, $\kax$ is the one-dimensional $k$-module spanned by the image $\hx$ of $x$ in $k_\alpha G$, and the action of $C_G(x)$ on $\kax$ is given by $g\cdot\hx=\alpha(g,x)\alpha(x,g)^{-1}\hx$.
\end{Proposition}

Note that setting $[\alpha]=[1]\in H^2(G;k^\times)$ above recovers the standard ``untwisted" centraliser decomposition. We remark that the action of $C_G(x)$ on $k_\alpha\hx$ is the ``twisted conjugation" action: $\hg\hx=\alpha(g,x)\widehat{gx}=\alpha(g,x)\widehat{xg}=\alpha(g,x)\alpha(x,g)^{-1}\hx\hg$, whence $g\cdot\hx=\hg\hx\hg^{-1}$ for all $g\in C_G(x)$. We will denote by $\lmd_\alpha:C_G(x)\to k^\times$, the homomorphism inducing this action defined by $g\mapsto\alpha(g,x)\alpha(x,g)^{-1}$ for all $g\in C_G(x)$. \begin{Definition} Following well known terminology \cite{Humph,Kar}, we say that an element $x\in G$ is $\alpha$-\textit{regular} if $\alpha(g,x)=\alpha(x,g)$ for all $g\in C_G(x)$.
\end{Definition}
Note that this occurs precisely when $C_G(x)=\ker(\lmd_\alpha)$, or equivalently $H^1(C_G(x);k_\alpha\hx)\cong \Hom(C_G(x),k)$, which will allow us to prove Corollary \ref{mainnonvanishing}. That is, showing the nonvanishing Hochschild cohomology of twisted finite simple group algebras will follow from determining the existence an $\alpha$-regular, weak Non-Schur element $x$ in each finite simple group; the existence of such a weak Non-Schur element $x$ is the main result of \cite{FJL}, whence it remains for us to show that these elements are $\alpha$-regular.

For the remainder of this section we fix the following: $x$ is an element of $G$, and $\alpha\in Z^2(G;k^\times)$ is such  that the cohomology class $[\alpha]$ corresponds to a central extension $1\to Z\to \hG\to G\to 1$. We also denote by $C_{\hG}(x)\leq \hG$ the preimage of $C_G(x)$ under the surjection $\hG\to G$. Note that two preimages $\hg,\hh\in \hG$ of an element $g\in G$ will differ by an element of $Z$, that $C_{\hG}(\hx)$ is normal in $C_{\hG}(x)$, and that we have a well-defined action of $C_{\hG}(x)$ on $\kax$ given by $\hg\cdot \hx=\lmd_\alpha(g)\hx$ for all $\hg\in C_{\hG}(x)$, which is trivial on restriction to $C_{\hG}(\hx)$. Let $N=\ker(\lmd_\alpha)=\{g\in C_G(x) \mid \alpha(g,x)=\alpha(x,g)\}$, and notice that $C_{\hG}(\hx)=\{\hg\in \hG \mid \alpha(g,x)=\alpha(x,g)\}$.

\begin{Proposition}\label{aregprops}
With the notation above, we have the following.

\begin{itemize}
    \item [(i)] There are isomorphisms of groups $C_{\hG}(\hx)/Z\cong N$, and $C_{\hG}(x)/Z\cong C_G(x)$.
    
    \item [(ii)] There is an equality of subgroups $C_{\hG}(\hx)=C_{\hG}(x)$ if and only if $x$ is $\alpha$-regular.
    
    \item [(iii)] The group $H:=C_{\hG}(x)/C_{\hG}(\hx)$ is isomorphic to a cyclic $p'$-subgroup of $k^\times$.
    
    \item [(iv)] There are isomorphisms of $k$-modules $$H^1(C_G(x);\kax)\cong H^1(C_{\hG}(x);\kax)\cong H^1(C_{\hG}(\hx);\kax)^H\cong H^1(N;\kax)^H,$$ where the final two terms are the fixed points under the action of $H$ induced by $C_{\hG}(x)$ on $\kax$.
\end{itemize}
\end{Proposition}

\begin{Remark}
The action of $H$ on $H^1(C_{\hG}(\hx);\kax)$ is induced by the action of $C_{\hG}(x)$ on $\kax$: ${^{\hg}}f:\ha\mapsto \lmd_\alpha(g)f(\hg^{-1}\ha\hg)$ for all $\hg\in C_{\hG}(x)$, $\ha\in C_{\hG}(\hx)$ and $f\in H^1(C_{\hG}(\hx);\kax)$, with the action on $H^1(N;\kax)$ defined similarly. Note that $\Hom(C_{\hG}(\hx),k)\cong H^1(C_{\hG}(\hx);\kax)\cong H^1(N;\kax)\cong \Hom(N,k)$, though we keep the notation as in the statement of the proposition because in each case $H$ acts possibly non-trivially. 
\end{Remark}

\noindent For the remainder of this section we fix $H$ as defined in the statement of the proposition.

\begin{proof}
On restriction to $C_{\hG}(\hx)$ and $C_{\hG}(x)$, the surjection $\hG\to G$ with kernel $Z$ gives, respectively, the desired isomorphisms in (i), from which (ii) is immediate (alternatively note that the action of $C_{\hG}(\hx)$ on $k_\alpha\hx$ is trivial). The third isomorphism theorem proves (iii), noting that $H\cong C_G(x)/N$ injects into $k^\times$. 

To ease notation, for the remainder of the proof we will write $k\hx=\kax$. To prove (v), we use the fundamental exact sequence \cite[Corollary 7.2.3]{Evens} induced by the short exact sequence $1\to C_{\hG}(\hx)\to C_{\hG}(x)\to H\to 1$ to see that \begin{equation}\label{FES} 0\to H^1(H;(k\hx)^{C_{\hG}(\hx)})\to H^1(C_{\hG}(x);k\hx)\to H^1(C_{\hG}(\hx);k\hx)^{H}\to H^2(H;(k\hx)^{C_{\hG}(\hx)})\to\cdots \ .\end{equation} As $H$ is cyclic of order coprime to the exponent of $k\hx$ (viewed as an additive group) and $C_{\hG}(\hx)$ acts trivially on $k\hx$, one sees that the second and fifth terms in (\ref{FES}) vanish, showing that $H^1(C_{\hG}(x);k\hx)\cong H^1(C_{\hG}(\hx);k\hx)^{H}$. The fundamental exact sequence induced by $1\to N\to C_G(x)\to H\to 1$, that is, on replacing the groups $C_{\hG}(\hx)$ with $N$ and $C_{\hG}(x)$ with $C_G(x)$ in (\ref{FES}), shows that $H^1(C_{G}(x);k\hx)\cong H^1(N;k\hx)^{H}$. We make one further use of the fundamental exact sequence, this time induced by the short exact sequence $1\to Z\to C_{\hG}(x)\to C_G(x)\to 1$: $$0\to H^1(C_G(x);(k\hx)^Z)\to H^1(C_{\hG}(x);k\hx)\to H^1(Z;k\hx)^{C_G(x)}\to \cdots \ , $$ which, since $Z$ is of $p'$-order and acts trivially on $k\hx$, shows that $H^1(C_G(x);k\hx)\cong H^1(C_{\hG}(x);k\hx)$, completing the proof.
\end{proof}

\begin{Lemma}\label{alpha}
Suppose that $x$ has order coprime to $|Z|$. Then $x$ is $\alpha$-regular, and the action of $C_G(x)$ on $\kax$ in Proposition \ref{CD} is trivial. In particular, any $p$-element is $\alpha$-regular.
\end{Lemma}

\begin{proof}
The group $Z$ is isomorphic to a cyclic subgroup of $k^\times$ of order $m$, say, and suppose that the order of $x$ is $r$. Note that $m$ is coprime to $p$. Consider the subgroup of $\hG$, $S:=\langle \hx, Z\rangle$. Then $S$ is an extension of $Z$ by a group of coprime index, $1\to Z\to S\to \langle x\rangle\to 1$. By the Schur-Zassenhaus Lemma this extension splits, $S\cong\langle x\rangle \rtimes Z$ and in fact equals the trivial extension, since $S$ is abelian. Whence we may choose  a lift $\hx\in \hG$ of the same order as $x$.

Now let $g\in C_G(x)$, so that $\hg\hx\hg^{-1}=z\hx$ for some $z\in Z$. With $m=|Z|$ one sees that $\hg\hx^m\hg^{-1}=\hx^m$, and $\hg\in C_{\hG}(\hx^m)$. On the other hand, $m$ is coprime to the order of $x$, so that $\langle \hx^m \rangle = \langle \hx \rangle$. Since $\hg\in C_{\hG}(\hx^m)$, $\hg$ commutes with all $y\in \langle \hx^m\rangle=\langle \hx \rangle$,  whence $C_{\hG}(\hx^m)=C_{\hG}(\hx)$. We have that $\hg\hx=\hx\hg$ and the result follows.
\end{proof}

\begin{Lemma}\label{cyclic}
Let $x$ be such that $C_G(x)$ is at least one of the following: a cyclic group, a $p$-group or a group whose Sylow subgroups are all cyclic. Then $x$ is $\alpha$-regular, and in particular the action of $C_G(x)$ on $\kax$ in Proposition \ref{CD} is trivial.
\end{Lemma}

\begin{proof}
 In each case the restriction of $\alpha$ to $C_G(x)$ is a coboundary, whence the result. 
\end{proof}

\begin{Proposition}

Suppose $x$ is such that its centraliser $C_G(x)$ is abelian of order divisible by $p$. Then  $H^1(C_G(x);\kax)\neq\{0\}$ if and only if $x$ is $\alpha$-regular.
\end{Proposition}

\begin{proof}
Note that if $C_G(x)$ is abelian of order divisible by $p$, then $\Hom(C_G(x),k)\neq\{0\}$. In one direction the proof is now clear: if the action of $C_G(x)$ is trivial, then $H^1(C_G(x);\kax)\cong\Hom(C_G(x),k)$.

In the other direction, suppose that $H^1(C_G(x);\kax)\neq\{0\}$, then via the isomorphisms of Proposition \ref{aregprops}(v) we have $H^1(N;\kax)^{H}\neq\{0\}$. Choose a non-zero $1$-cocycle $f:N\to \kax$. As $C_G(x)$ is abelian, we have that $f$ is a fixed by the action of $H$ if and only if  $\lmd_\alpha(g)f(a)=f(a)$ for all $g\in C_G(x)$ and $a\in N$. As $f$ is non-zero one sees that this can occur if and only if the action of $C_G(x)$ is trivial. This completes the proof.
\end{proof}

\begin{Remark}
From this proof it is easy to see that more generally, for $C_G(x)$ not necessarily abelian it is still the case that for all $f\in Z^1(C_G(x);\kax)$, $f\equiv 0$ on $Z(C_G(x))$.
\end{Remark}

\begin{Proposition}\label{chgab}
Let $x$ be a weak Non-Schur element, let $C_{\hG}(\hx)$ be abelian, and suppose that $Z$ is (cyclic) of prime order. Then $x$ is $\alpha$-regular, and $HH^1(k_\alpha G)\neq \{0\}$.
\end{Proposition}

\begin{proof}
Note that if the $p$-part of $x$ is not in $C_G(x)'$, then $C_{\hG}(\hx)$ has order divisible by $p$. Since $C_{\hG}(\hx)$ is abelian then by Proposition \ref{aregprops}(ii), $N$ must be abelian. What is more, all Sylow $\ell$-subgroups of $C_{\hG}(x)$ and $C_G(x)$ are abelian, except for possibly when $\ell=|Z|$ (see for example \cite[Proposition 2.2]{Humph}). Recall that $|Z|$ is coprime to $p$.

We may assume that $H$ is non-trivial else $x$ is $\alpha$-regular by definition. Since $\alpha$ takes values in $Z$, we have that $H\cong Z\cong C_\ell$. Whence $C_{\hG}(x)$ and $C_G(x)$ are metabelian, and in particular they are each extensions of a cyclic group by an abelian group (and therefore also solvable).

Let $Q$ be the Sylow $\ell$-subgroup of $N$, $S$ a Sylow $\ell$-subgroup of $C_G(x)$ (necessarily containing $Q$), and write $N=M\times Q$ where $M$ is the largest $\ell'$-subgroup of $N$. Then $M$ is normal in $C_G(x)$ and by the Schur-Zassenhaus theorem we have that $C_G(x)=M\rtimes S$. Whence $x$ is either in $M$ or $S$. If $x\in M\subseteq N$ then $x$ is $\alpha$-regular and we are done. Otherwise $x\in S$, which forces the $p$-part of $x$ to be trivial - as $\ell$ is coprime to $p$ - and thus contained in $C_G(x)'$, in contradiction to the hypothesis. Thus $x$ is $\alpha$-regular, $H^1(C_G(x);\kax)\cong \Hom(C_G(x),k)$ by Lemma \ref{alpha}, and the result follows by Proposition \ref{CD}.
\end{proof}

\begin{Definition} Following the terminology of \cite{FJL}, we make the following definitions.\begin{itemize}
    \item [(1)] We say that $G$ satisfies the \textit{weak Non-Schur property} $W(p)$ if there is some $x\in G$ such that its $p$-part is not contained in $C_G(x)'$.
    
    \item[(2)] We say that $G$ satisfies the \textit{strong Non-Schur property} $S(p)$ if there is a $p$-element $x\in G$ such that $x$ is not contained in $C_G(x)'$, and we call such an element $x$ a \textit{strong Non-Schur element}.
\end{itemize}
\end{Definition}

\noindent Evidently we have $S(p)$ implies $W(p)$. What is more, the main result of \cite{FJL} proves that $W(p)$ holds for all finite groups $G$. As noted in the introduction, this may be used in conjunction with Proposition \ref{CD}, setting $[\alpha]=[1]$ in $H^2(G;k^\times)$, to show that $HH^1(kG)\neq\{0\}$ for all finite groups. We record this fact here.

\begin{Proposition}\label{nonvanishingkg}
Let $G$ satisfy $W(p)$ for some weak Non-Schur element $x$. Then $\Hom(C_G(x),k)$ is nonzero. In particular, $HH^1(kG)\neq\{0\}$ for all finite groups $G$.
\end{Proposition}

\begin{proof}
Recall that $O^p(G)$ denotes the smallest normal subgroup of $G$ such that the quotient group $G/O^p(G)$ is a $p$-group. Let $A=C_G(x)/C_G(x)'$, which is of order divisible by $p$, since $G$ satisfies $W(p)$. Then with $R=A/O^p(A)$, the quotient $R/\Phi(R)$ is a nontrivial elementary abelian $p$-group with rank equal to $\dim_k(\Hom(C_G(x),k))$. The second statement then follows from \cite{FJL} and Proposition \ref{CD}. 
\end{proof}

It is not true in general that $W(p)$ implies $S(p)$, as shown in \cite{FJL}. Indeed, were this the case then one could easily show that $HH^1(k_\alpha G)\neq\{0\}$ for \textit{all} finite groups $G$ and all $\alpha\in Z^2(G;k^\times)$, via the following result.

\begin{Lemma}\label{nonvanishingsp}
Let $G$ satisfy $S(p)$. Then for all $\alpha\in Z^2(G;k^\times)$, every strong Non-Schur element $x$ is $\alpha$-regular. In particular,  $HH^1(k_\alpha G)\neq\{0\}$.
\end{Lemma} 

\begin{proof}
On the one hand, $x$ is a $p$-element satisfying $\Hom(C_G(x),k)\neq\{0\}$ by Proposition \ref{nonvanishingkg}. On the other, by Lemma \ref{alpha} $x$ is $\alpha$-regular, so $H^1(C_G(x);k_\alpha\hx)\cong \Hom(C_G(x),k)$, and the result follows by Proposition \ref{CD}.
\end{proof}

It is easy to show a slight reformulation of the proposition above: any group satisfying $S(p)$ also satisfies $\dim_k(HH^1(k\hG))>\dim_k(HH^1(kG))$, so that the left hand side of (\ref{dimsum}) is nonzero.

One notes that it is somewhat tautological to refer to group elements as \textit{strong Non-Schur $\alpha$-regular} elements, though we will proceed to do so as it distinguishes this more powerful property from the weak Non-Schur $\alpha$-regular case. 

\begin{Lemma}\cite[Lemma 1.2]{FJL}\label{absp}
Let $P$ be a Sylow $p$-subgroup of $G$ and suppose that $Z(P)$ is not contained in $P'$. Then $G$ satisfies $S(p)$. In particular, $G$ satisfies $S(p)$ if it has abelian Sylow $p$-subgroups.
\end{Lemma}

\section{Proof of Theorem \ref{main}}\label{sec3}

Throughout this section we fix $k$ to be an algebraically closed field of characteristic $p$ dividing the order of a finite group $G$; all notation is as in Sections \ref{sec1} and \ref{sec2}. We begin with abelian groups.

\begin{Lemma}\label{abareg}
Let $G$ be an abelian group, $x\in G$ a $p$-element. Then $x$ is an $\alpha$-regular, strong Non-Schur element of $G$. In particular, $HH^1(k_\alpha G)\neq \{0\}$.
\end{Lemma}

\begin{proof}
Such an $x$ is a strong Non-Schur element by virtue of the fact that $C_G(x)'=G'=\{1\}$. The result follows from Lemma \ref{nonvanishingsp}
\end{proof}

\noindent In fact, we can show something a little stronger.

\begin{Theorem}\label{abelian}
Suppose $O^p(G)$ is a proper subgroup of $G$. Then $HH^1(k_\alpha G)\neq \{0\}$ for all $\alpha \in Z^2(G;k^\times)$.
\end{Theorem}

\begin{proof}
Let $R=G/O^p(G)$ be the largest $p$-group quotient of $G$, then $R/\Phi(R)$ is a non-trivial elementary abelian $p$-group. Whence $\dim_k(\Hom(G,k))=\dim_k(\Hom(R/\Phi(R),k))\neq0$. Any cocycle $\alpha\in H^2(G;k^\times)$ satisfies $\alpha(g,1)=\alpha(1,1)=\alpha(1,g)$ for all $g\in G$ \cite[Proposition 1.2.5]{LinBlockI}, so on choosing the conjugacy class representative $x=1$ in the decomposition of Proposition \ref{CD}, one sees that the action of $G=C_G(x)$ on $\kax$ is trivial. The result follows. 
\end{proof}

\begin{Proposition}\label{S(p)}
Let $G$ be a finite simple group of order divisible by $p$ and $q$ a prime power. Suppose further that $G$ is one of the following.
\begin{itemize}
    \item An alternating group $A_n$ for $n\geq 5$.
    \item A sporadic group.
    \item A Chevalley group of classical type, $Gl_n(q),Sl_n(q),PSl_n(q),U_n(q),SU_n(q),PSU_n(q), \\ SO_{2n+1}(q), P\Omega_{2n+1}(q),Sp_{2n}(q),PSp_{2n}(q),SO_{2n}(q),P\Omega_{2n}(q),SO_{2n}^-(q)$ or $P\Omega_{2n}^-(q)$.
\end{itemize}
Then for all $\alpha \in Z^2(G;k^\times)$, $G$ contains weak Non-Schur $\alpha$-regular elements, and in particular $HH^1(k_\alpha G)\neq \{0\}$.
\end{Proposition}

\begin{proof}
This is an immediate consequence of \cite[Propositions 2.1,2.2 and 3.2]{FJL}, and in fact, aside from the exceptions $(G,p)\in \{(Ru,3),(J_4,3),(Th,5)\}$ these groups all satisfy $S(p)$, and so the result holds in these cases.

For the exceptions we have the following: the Schur multiplier of $J_4$ and $Th$ is trivial whence $H^2(G;k^\times)$ is also, and there is nothing to show. For the case $G=Ru$, one can find an element $x\in G$ of order divisible by $p$ such that $C_G(x)=\langle x \rangle$ \cite[Proposition 2.2]{FJL}. The result follows from Lemma \ref{cyclic}.

\end{proof}

We now turn to the finite groups Lie type. 

\begin{Lemma}\label{trivschur}
Let $G$ be one of $^2G_2(3^{2n+1})$ for some $n\geq 1$, the simple group $^2G_2(3)'$, $^3D_4(q^3)$, $^2F_4(q^2)$, $E_8(q)$ or the simple Tits group $^2F_4(2)'$, for some prime power $q$. Then for all $\alpha\in Z^2(G;k^\times)$, $G$ has strong Non-Schur $\alpha$-regular elements, and in particular $HH^1(k_\alpha G)\neq \{0\}$.
\end{Lemma}

\begin{proof}
One sees this by noting that in all cases the Schur multiplier is trivial. 
\end{proof}

\begin{Proposition}\label{LieDiv}
Let $G$ be a finite group of Lie type, defined over a field of characteristic $r$ with $q=r^m$ for some $m$. Let $\alpha \in Z^2(G;k^\times)$. Then we have the following.
\begin{itemize}
    \item [(i)] If $p=r$ then $G$ contains strong Non-Schur $\alpha$-regular elements.
    
    \item [(ii)] Suppose that $p\neq r$. If $W$ is a Weyl group of $G$ and $p$ does not divide the order of $W$, then $G$ contains strong Non-Schur $\alpha$-regular elements.
\end{itemize}
In each case, $HH^1(k_\alpha G)\neq \{0\}$.
\end{Proposition}

\begin{proof}
This is an immediate consequence of \cite{FJL}: if $p=r$ then by \cite{BCCISS} and \cite[Lemma 3.1(1)]{FJL} $G$ contains a $p$-element $x$ whose centraliser is abelian of order divisible by $p$, and we are done by Lemma \ref{nonvanishingsp}. If $p\neq r$, and $p$ does not divide $|W|$, then by \cite{BCCISS} and \cite[Lemma 3.1(2)]{FJL} the Sylow $p$-subgroups of $G$ are abelian, and we are done by Lemma \ref{absp}.
\end{proof}

\begin{Proposition}\label{suz}
Let $G$ be a Suzuki group $^2B_2(2^{2n+1})$ for some $n\geq 1$. Then for all $\alpha \in Z^2(G;k^\times)$, $G$ contains strong Non-Schur $\alpha$-regular elements, and in particular $HH^1(k_\alpha G)\neq \{0\}$.
\end{Proposition}

\begin{proof}
If $n>1$ then the Schur multiplier of $G$ is trivial and there is nothing to show. If $n=1$ then $G={^2B_2(8)}$ has Schur multiplier isomorphic to the Klein-four group. We may assume by Proposition \ref{LieDiv} that $p\neq 2$ and $p$ divides the order of a Weyl group $W$ of $G$. On the other hand, the Weyl group of a root system of $B_2$ type has order $8$, completing the proof.
\end{proof}

\begin{Proposition}\label{G2}
Let $G$ be a Chevalley group $G_2(q)$ for some prime power $q>2$, or the simple group $G_2(2)'$. Then for all $\alpha \in Z^2(G;k^\times)$ $G$ contains weak Non-Schur $\alpha$-regular elements, and in particular $HH^1(k_\alpha G)\neq \{0\}$. 
\end{Proposition}

\begin{proof}
There is an isomorphism $G_2(2)'\cong PSU_3(3^2)$ and so we are done in this case by Proposition \ref{S(p)}. Now suppose $G=G_2(q)$, $q>2$. A Weyl group of $G$ has order $2^2\cdot 3$ \cite[Lemma 3.1]{FJL} so by Lemma \ref{LieDiv} we only need consider $p=2,3$. The Schur multiplier of $G$ is trivial except when $q=3$ or $4$, in which case it is of order $3$ or $2$ respectively.

First let $q=3$. By Proposition \ref{LieDiv} we have $p=2$ (otherwise we are done). Let $\alpha \in Z^2(G;k^\times)$. One verifies in \texttt{GAP} that $G$ has an element $x$ of order $8$, with centraliser $C_G(x)=\langle x \rangle$, and we are done by Lemma \ref{cyclic}: $G$ satisfies $S(p)$.

Now let $q=4$, so that $p=3$. In this case, using \texttt{GAP} one can find an element $x$ of order $15$ with $C_G(x)=\langle x \rangle$, and we are again done by Lemma \ref{cyclic}, though $G$ only satisfies $W(p)$ in this case.
\end{proof}

\noindent Note that these calculations agree with the more general calculations done by Fleischmann, \\ Janiszczak and Lempken in \cite[Proposition 4.1]{FJL}.

\begin{Proposition}\label{F4}
Let $G$ be a Chevalley group $F_4(q)$ for some prime power $q$. Then for all $\alpha\in Z^2(G;k^\times)$, $G$ contains weak Non-Schur $\alpha$-regular elements, and in particular $HH^1(k_\alpha G)\neq \{0\}$.
\end{Proposition}

\begin{proof}
The Schur multiplier of $G$ is trivial in all cases except $q=2$. A Weyl group of $F_4$ type has order $2^7\cdot3$, so by Lemma \ref{LieDiv} we only need to check $G=F_4(2)$ and $p=3$. By \cite[Lemma 3.1]{FJL} $C_G(x)$ has order $21$ and in fact one verifies (for example in \texttt{GAP}) that $C_G(x)\cong C_{21}$. By Lemma \ref{cyclic} the action of $C_G(x)$ on $\kax$ is trivial. This completes the proof.
\end{proof}

\begin{Proposition}\label{e67}
Let $q$ be a nontrivial prime power and $G$ be one of the exceptional Chevalley groups $E_6(q)$, ${^2}E_6(q^2)$ or $E_7(q)$. Then for all $\alpha \in Z^2(G;k^\times)$, $G$ contains weak Non-Schur $\alpha$-regular elements, and in particular $HH^1(k_\alpha G)\neq \{0\}$.
\end{Proposition}

\begin{proof}
This follows from Proposition \ref{chgab}. Let $G$ be as in the statement of the proposition and $\hG$ the central extension of $G$ by $Z$ representing $\alpha$. Then $Z$ is cyclic of order $\gcd(q-1,3)$, $\gcd(q+1,3)$ and $\gcd(q-1,2)$ for $E_6(q)$, ${^2}E_6(q^2)$ and $E_7(q)$ respectively.

By \cite[Proposition 4.1]{FJL} we have that the necessary conditions of Proposition \ref{chgab} are satisfied for $G$, completing the proof.
\end{proof}

\begin{proof}[Proof of Theorem \ref{main}]
The proof now follows from Theorem \ref{abelian} and Propositions \ref{S(p)},  \ref{suz}, \ref{G2}, \ref{F4}, \ref{e67} and Lemma \ref{trivschur}.
\end{proof}

\begin{Remark}
It is relatively straightforward to verify that in fact, every finite simple group satisfies either $S(p)$ or contains an element $x$ satisfying the hypothesis of Proposition \ref{chgab}, providing an alternative proof of Theorem \ref{main}.
\end{Remark}

\begin{proof}[Proof of Corollary \ref{mainnonvanishing}]
By Theorem \ref{main}, $G$ contains weak Non-Schur $\alpha$-regular elements. On the one hand such an element $x$ gives $\Hom(C_G(x),k)\neq\{0\}$ by Proposition \ref{nonvanishingkg}, and on the other $\Hom(C_G(x),k)\cong H^1(C_G(x);k_\alpha\hx)$ since $x$ is $\alpha$-regular. The result now follows from the twisted centraliser decomposition, Proposition \ref{CD}.
\end{proof}

\section{Dimensions}\label{sec4}

Using the isomorphism given by (\ref{(2)}) and the comments that follow it, the centraliser decomposition, and the \texttt{GAP} code in \cite[Appendix A]{MurMat} we are able to calculate the dimensions of the first Hochschild cohomology of some finite simple group algebras. Since the cover $\hG$ is an extension of $G$ by a relatively small central subgroup $Z$ in each case below, finding the dimensions of $HH^1(k_\alpha G)$ is a case of elementary arithmetic.

More explicitly, we have that for the case where $Z=C_2$, $k\hG$ is isomorphic as a $k$-algebra to $kG\times k_\alpha G$. Taking Hochschild cohomology, one obtains $HH^1(k\hG)\cong HH^1(kG)\oplus HH^1(k_\alpha G)$  which may then be used to show that $\dim_k(HH^1(k_\alpha G)=\dim_k(HH^1(k\hG))-\dim_k(HH^1(kG))$. When $Z=C_3$, we have $kG\times k_\alpha G\times k_{\alpha^{-1}}G$ so that $HH^1(kG)\oplus HH^1(k_\alpha G)^{\oplus 2}$ by our earlier remarks, whence $\dim_k(HH^1(k_\alpha G)=(\dim_k(HH^1(k\hG))-\dim_k(HH^1(kG)))/2$.

When $Z=C_4$, $C_6$ or in one case $C_{12}$ we have to be a little more careful.  In some of these cases, one finds that by Corollary \ref{mainnonvanishing} the only option is that all the dimensions of the $HH^1(k_{\alpha^i}G)$, $i=1,\ldots,|Z|-1$ are all equal to $1$; take, for example, the case where $G=A_7$, $Z=C_6$ and $p=5$. Here, $\dim_k(HH^1(k\hG))=6$, $\dim_k(HH^1(kG))=1$ so that $\sum_{i=1}^5\dim_k(HH^1(k_{\alpha^i}G))=5$. By our main corollary, the nonvanishing of $HH^1(k_\alpha G)$, one therefore obtains $\dim_k(HH^1(k_{\alpha^i}G))=1$ for $i=0,\ldots,5$. On the other hand, in this same example with instead the prime $p=7$, we have that $\alpha^2$ and $\alpha^{-2}$ correspond exactly to the entry of the table for which $Z=C_3$, and $\alpha^3$ corresponds to the earlier entry for which $Z=C_2$. In other words, subtracting the dimensions of $\dim_k(HH^1(k_\alpha G))$ for the entries $Z=C_3$ and $Z=C_2$ from the dimension of $HH^1(k\hG)$ we conclude that the dimensions of $HH^1(k_{\alpha^i}G)$, $i=0,\ldots,5$ are equal to $2$. 

In fact, the only occurrence in our table for which these dimensions are not all equal for all powers of $\alpha$ is when $G$ is the Mathieu group $M_{22}$, $Z=C_4$ and $p=3$. In this case, one sees (using the same reasoning as above) that $\dim_k(HH^1(k_{\alpha^i}G)$ is equal to $3,2,3,2$ for $i=0,1,2,3$ respectively.

Our list was chosen by working our way through the groups whose character tables and further details are given in The ATLAS \cite{ATLAS}, discarding those groups which \texttt{GAP} could not construct nor perform the necessary calculations on. The group notation used is as in The ATLAS; all other notation follows as previously.

\begin{Proposition}\label{dimsprop}
Let $G$ be one of the groups given in the first column of Table \ref{dims}, and $\hG$ be a central extension of $G$ by a cyclic $p'$-group $Z$. Then for all $\alpha \in Z^2(G;k^\times)$ corresponding to a faithful character $Z\to k^\times$, the dimensions of the first Hochschild cohomology of $k_\alpha G$ are given as in the final column of Table \ref{dims}.
\end{Proposition}
\newpage
\begin{center}
\begin{longtable}{c|c|c|c|c|c|c}
\caption{ Dimensions of the first Hochschild cohomology of some finite simple group algebras   }\label{dims} \\
		$G$	    	& $M(G)$	& $Z$     & $p$	    & $\dim_k(HH^1(k\hG))$ & $\dim_k(HH^1(kG))$ & $\dim_k(HH^1(k_\alpha G))$ \\
		\hline
		\hline
		$A_5$       & $C_2$     & $C_2$   & $3$     & $2$   & $1$   & $1$   \\
		            &           &         & $5$     & $4$   & $2$   & $2$   \\
		$L_3(2)$    & $C_2$     & $C_2$   & $3$     & $2$   & $1$   & $1$   \\
		            &           &         & $7$     & $4$   & $2$   & $2$   \\
		$A_6$       & $C_6$     & $C_2$   & $3$     & $8$   & $4$   & $4$   \\	
		            &           &         & $5$     & $4$   & $2$   & $2$   \\
		            &           & $C_3$   & $2$     & $9$   & $3$   & $3$   \\
		            &           &         & $5$     & $6$   & $2$   & $2$   \\
		            &           & $C_6$   & $5$     & $12$  & $2$   & $2$   \\		            
		$L_2(11)$   & $C_2$     & $C_2$   & $3$     & $5$   & $2$   & $3$   \\
		            &           &         & $5$     & $4$   & $2$   & $2$   \\		
		            &           &         & $7$     & $4$   & $2$   & $2$   \\
		$L_2(13)$   & $C_2$     & $C_2$   & $3$     & $5$   & $2$   & $3$   \\            
		            &           &         & $7$     & $6$   & $3$   & $3$   \\
		            &           &         & $13$    & $4$   & $2$   & $2$   \\		      
		$L_2(17)$   & $C_2$     & $C_2$   & $3$     & $8$   & $4$   & $4$   \\
		            &           &         & $17$    & $4$   & $2$   & $2$   \\ 
		$A_7$       & $C_6$     & $C_2$   & $3$     & $9$   & $5$   & $4$   \\	
		            &           &         & $5$     & $2$   & $1$   & $1$   \\
		            &           &         & $7$     & $4$   & $2$   & $2$   \\		            
		            &           & $C_3$   & $2$     & $17$  & $5$   & $6$   \\
		            &           &         & $5$     & $3$   & $1$   & $1$   \\
		            &           &         & $7$     & $6$   & $2$   & $2$   \\		            
		            &           & $C_6$   & $5$     & $6$   & $1$   & $1$   \\
		            &           &         & $7$     & $12$  & $2$   & $2$  \\		            
		$L_2(19)$   & $C_2$     & $C_2$   & $3$     & $8$   & $4$   & $4$   \\
		            &           &         & $5$     & $9$   & $4$   & $5$   \\
		            &           &         & $19$    & $4$   & $2$   & $2$   \\ 

		$L_2(23)$   & $C_2$     & $C_2$   & $3$     & $11$  & $5$   & $6$   \\
		            &           &         & $11$    & $10$  & $5$   & $5$   \\
		            &           &         & $23$    & $4$   & $2$   & $2$   \\ 		    
		$L_2(27)$   & $C_2$     & $C_2$   & $3$     & $12$  & $6$   & $6$   \\
		            &           &         & $7$     & $13$  & $6$   & $7$   \\
		            &           &         & $13$    & $12$  & $6$   & $6$   \\ 		
		$L_2(29)$   & $C_2$     & $C_2$   & $3$     & $14$  & $7$   & $7$   \\
		            &           &         & $5$     & $14$  & $7$   & $7$   \\
		            &           &         & $7$     & $13$  & $6$   & $7$   \\
		            &           &         & $29$    & $4$   & $2$   & $2$   \\ 	
		$L_2(31)$   & $C_2$     & $C_2$   & $3$     & $14$  & $7$   & $7$   \\
		            &           &         & $5$     & $14$  & $7$   & $7$   \\
		            &           &         & $31$    & $4$   & $2$   & $2$   \\
		            		             
		            		            \pagebreak
		$G$	    	& $M(G)$	& $Z$     & $p$	    & $\dim_k(HH^1(k\hG))$ & $\dim_k(HH^1(kG))$ & $\dim_k(HH^1(k_\alpha G))$ \\
		\hline
		\hline
		$A_8$       & $C_2$     & $C_2$   & $3$     & $14$  & $7$   & $7$   \\
		            &           &         & $5$     & $6$   & $3$   & $3$   \\
		            &           &         & $7$     & $4$   & $2$   & $2$   \\
		$L_3(4)$    & $C_4^2\times C_3$   & $C_2$   & $3$   & $4$   & $2$   & $2$    \\
		            &           &         & $5$     & $4$   & $2$   & $2$   \\
		            &           &         & $7$     & $4$   & $2$   & $2$   \\
		            &           & $C_3$   & $2$     & $30$  & $10$ & $10$   \\
		            &           &         & $5$     & $6$   & $2$   & $2$   \\	
		            &           &         & $7$     & $6$   & $2$   & $2$   \\
		            &           & $C_6$   & $5$     & $12$  & $2$   & $2$   \\
		            &           &         & $7$     & $12$  & $2$   & $2$   \\			            
		$U_4(2)$    & $C_2$     & $C_2$   & $3$     & $39$  & $20$  & $19$  \\
		            &           &         & $5$     & $2$   & $1$   & $1$   \\
		$Sz(8)$     & $C_2^2$   & $C_2$   & $5$     & $2$   & $1$   & $1$   \\
		            &           &         & $7$     & $6$   & $3$   & $3$   \\
		            &           &         & $13$    & $6$   & $3$   & $3$   \\
        $M_{12}$    & $C_2$     & $C_2$   & $3$     & $7$   & $4$   & $3$   \\
		            &           &         & $5$     & $4$   & $2$   & $2$   \\
		            &           &         & $11$    & $4$   & $2$   & $2$   \\
		$A_9$       & $C_2$     & $C_2$   & $3$     & $25$  & $12$  & $13$  \\
		            &           &         & $5$     & $7$   & $4$   & $3$   \\
		            &           &         & $7$     & $2$   & $1$   & $1$   \\
		$J_2$       & $C_2$     & $C_2$   & $3$     & $13$  & $7$   & $6$   \\
		            &           &         & $5$     & $18$  & $10$  & $8$   \\
		            &           &         & $7$     & $2$   & $1$   & $1$   \\ 
		$S_6(2)$    & $C_2$     & $C_2$   & $3$     & $28$  & $16$  & $12$  \\
		            &           &         & $5$     & $6$   & $3$   & $3$   \\
		            &           &         & $7$     & $2$   & $1$   & $1$   \\
		$A_{10}$    & $C_2$     & $C_2$   & $3$     & $29$  & $15$  & $14$  \\
		            &           &         & $5$     & $9$   & $5$   & $4$   \\
		            &           &         & $7$     & $6$   & $3$   & $3$   \\
		$L_3(7)$    & $C_3$     & $C_3$   & $2$     & $38$  & $12$  & $13$  \\
		            &           &         & $7$     & $21$  & $7$   & $7$   \\
		            &           &         & $19$    & $18$  & $6$   & $6$   \\
        $G_2(3)$    & $C_3$     & $C_3$   & $2$     & $33$  & $11$  & $11$   \\
		            &           &         & $7$     & $3$   & $1$   & $1$   \\
		            &           &         & $13$    & $6$   & $2$   & $2$   \\
		$S_4(5)$    & $C_2$     & $C_2$   & $3$     & $24$  & $12$  & $12$  \\
		            &           &         & $5$     & $39$  & $20$  & $19$   \\
		            &           &         & $13$    & $6$   & $3$   & $3$   \\		            
		            		            
		            \pagebreak
		$G$	    	& $M(G)$	& $Z$     & $p$	    & $\dim_k(HH^1(k\hG))$ & $\dim_k(HH^1(kG))$ & $\dim_k(HH^1(k_\alpha G))$ \\
		\hline
		\hline		            
        $M_{22}$    & $C_{12}$  & $C_2$   & $3$     & $6$   & $3$   & $3$   \\
		            &           &         & $5$     & $2$   & $1$   & $1$    \\
		            &           &         & $7$     & $4$   & $2$   & $2$    \\
		            &           &         & $11$    & $4$   & $2$   & $2$   \\	
		            &           & $C_3$   & $2$     & $29$  & $9$   & $10$   \\	
		            &           &         & $5$     & $3$   & $1$   & $1$   \\
		            &           &         & $7$     & $6$   & $2$   & $2$   \\		
		            &           &         & $11$    & $6$   & $2$   & $2$   \\		       
		            &           & $C_4$   & $3$     & $10$  & $3$   & $2$   \\	
		            &           &         & $5$     & $4$   & $1$   & $1$   \\
		            &           &         & $7$     & $8$   & $2$   & $2$   \\
		            &           &         & $11$    & $8$   & $2$   & $2$   \\	
		            &           & $C_6$   & $5$     & $6$   & $1$   & $1$   \\
		            &           &         & $7$     & $12$  & $2$   & $2$   \\
		            &           &         & $11$    & $12$  & $2$   & $2$   \\	
		            &           & $C_{12}$ & $5$     & $12$ & $1$   & $1$   \\
		            &           &         & $7$     & $24$  & $2$   & $2$   \\
		            &           &         & $11$     & $24$ & $2$   & $2$   \\

		$U_3(8)$    & $C_3$     & $C_3$   & $2$     & $45$  & $15$  & $15$  \\
		            &           &         & $7$     & $27$  & $9$   & $9$   \\
		            &           &         & $19$    & $18$  & $6$   & $6$   \\
		$A_{11}$    & $C_2$     & $C_2$   & $3$     & $37$  & $20$  & $17$  \\
		            &           &         & $5$     & $13$  & $7$   & $6$   \\
		            &           &         & $7$     & $7$   & $4$   & $3$   \\	          
		            &           &         & $11$    & $4$   & $2$   & $2$   \\
		$HS$        & $C_2$     & $C_2$   & $3$     & $10$  & $5$   & $5$  \\
		            &           &         & $5$     & $15$  & $8$   & $7$   \\
		            &           &         & $7$     & $2$   & $1$   & $1$   \\	 
		            &           &         & $11$    & $4$   & $2$   & $2$   \\
		$A_{12}$    & $C_2$     & $C_2$   & $3$     & $37$  & $29$  & $17$  \\
		            &           &         & $5$     & $13$  & $10$  & $6$   \\
		            &           &         & $7$     & $7$   & $5$   & $3$   \\	          
		            &           &         & $11$    & $4$   & $2$   & $2$   \\		            
		$G_2(4)$    & $C_2$     & $C_2$   & $3$     & $24$  & $14$  & $10$  \\
		            &           &         & $5$     & $22$  & $12$  & $10$   \\
		            &           &         & $7$     & $6$   & $3$   & $3$   \\	          
		            &           &         & $13$    & $4$   & $2$   & $2$   \\		            
		          
\end{longtable}
\end{center}

\end{document}